\documentclass[11pt]{article}
\usepackage{amsfonts}
\usepackage{latexsym}
\usepackage{amssymb}
\usepackage{epsf}
\usepackage{mathrsfs}
\usepackage{amsmath}
\usepackage{graphicx}
\usepackage{flafter}
\usepackage{mathbbold}
\newcommand{\ep}{\hfill\rule{0.15cm}{0.35cm}\vskip 0.3cm}

\newtheorem{theorem}{Theorem}[section]

\newtheorem{lemma}{Lemma}[section]

\newtheorem{remark}{Remark}[section]

\newenvironment{proof}[1][Proof]{\begin{trivlist}
\item[\hskip \labelsep {\bfseries #1}]}{\ep\end{trivlist}}

\begin{document}

\title{An Abstract Stabilization Method with Applications to
Nonlinear Incompressible Elasticity
}
\author{Qingguo Hong, Chunmei Liu,  Jinchao Xu \\
{\small Department of Mathematics, The Pennsylvania State University,}\\
{\small Unverisity Park, PA 16802, USA}\\
{\small e-mail: huq11@psu.edu}\\
{\small College of Science, Hunan University of Science and Engineering, }\\
{\small Yongzhou 425199, Hunan, People's Republic of China}\\
{\small e-mail: liuchunmei0629@163.com}\\
{\small Department of Mathematics, Pennsylvania State University,}\\
{\small University Park, PA 16802, USA}\\
{\small e-mail: xu@math.psu.edu}}
\date{}
\maketitle

\begin{abstract}
In this paper, we propose and analyze an abstract stabilized mixed finite element
framework that can be applied to nonlinear incompressible elasticity problems. In the abstract stabilized
framework, we prove that any mixed finite element method that satisfies the discrete inf-sup condition can be
modified so that it is stable and optimal convergent as long as the mixed continuous problem is stable. Furthermore,
we apply the abstract stabilized framework to nonlinear incompressible elasticity problems and present
numerical experiments to verify the theoretical results.
\end{abstract}


{\bf MR (2000) Subject Classifications: } 65N30

\section{Introduction}
\label{INTR}
\setcounter{equation}{0}
Many finite element schemes perform very well in
terms of both accuracy and stability for linear elastic
problems (see \cite{reddy1995stability,braess2004uniform,
houston2006hp,hansbo2002discontinuous,hong2016robust,wu2017interior, gong2019new, wang2020mixed, hong2020extended}), including in highly
constrained situations such as incompressible cases. However, it is
well-known that though such schemes can be extended to nonlinear
elasticity cases (for instance, large deformation elasticity problems) the same level of  stability is by no 
means guaranteed (see
 \cite{auricchio2005stability,wriggers1996note,lovadina2003enhanced,
pantuso1997stability}).

In \cite{eyck2008adaptive,eyck2008adaptiveM}, a stabilized
discontinuous Galerkin method for nonlinear compressible elasticity
was introduced. Because the stabilization term adapts to the solution of the problem by
locally changing the size of a penalty term on the appearance of discontinuities,
it is called an adaptive stabilization
strategy. This same adaptive stabilization strategy can also be
found in \cite{ten2010adaptive}. Although all three papers discuss compressible and nearly
incompressible nonlinear elasticity problems, they do not address the exactly
incompressible problems.

In \cite{auricchio2010importance}, an isogeometric ``stream function" formulation
was used to exactly enforce the linearized incompressibility constraint.
In fact, the exact satisfaction of the linearized incompressibility
constraint leads to a numerical approximation of the stability
range that converges to the exact one (i.e., the one for the continuum
problem). As the stream function is given by a fourth-order PDE,
the high regularity of the NURBS shape functions \cite{hughes2005isogeometric} 
must be used there and the stream function in 3dimensions is not unique.

Brezzi, Fortin, and Marini \cite{brezzimixed} presented a modified mixed formulation for second-order
elliptic equations and linear elasticity problems that automatically satisfied the coercivity condition
on the discrete level. Using the same technique and applying stable Stokes elements to a modified 
formulation for Darcy-Stokes-Brinkman models, Xie, Xu and Xue \cite{xie2008uniformly} 
obtained a uniformly stable method.

In this paper, motivated by the modified formulation used
in \cite{brezzimixed,xie2008uniformly} for second-order equations and the Darcy-Stokes-Brinkman
models respectively, we devise a stabilization strategy for the classical mixed finite method
designed for Stokes equations and obtain a modified method for nonlinear incompressible elasticity.
As indicated in \cite{auricchio2005stability}, a usual
mixed finite element method that is stable for Stokes equations such as the MINI element
is not stable for nonlinear incompressible elasticity
even though the continuous problem is stable.
We note that, for the usual mixed finite method, the unphysical instability
is caused by the fact that the discrete solution is not exactly divergence-free.
Hence, we rewrite the continuous problem.
Based on that we design a stabilization strategy involving the divergence of
the discrete solution. We prove that the modified mixed finite element  method can
remove the unphysical instability and lead to optimal convergence.
Hence, all stable Stokes elements can also be made stable for discrete
nonlinear elasticity problems. 

The rest of the paper is organized as follows.
In section \ref{TFSI}, we present the finite strain incompressible elasticity problem and
obtain the linearized formulation of the nonlinear elasticity problem.
In section \ref{ASOM}, we propose the stabilization strategy and derive the modified mixed finite
element method in an abstract framework, and in section \ref{ATTN} we apply
this modified method to the nonlinear incompressible elasticity problem.
Furthermore, we obtain an optimal approximation of the modified finite element
method. In section \ref{NUEX}, we give two simple examples and the corresponding numerical
experiments to verify the stability and accuracy of the modified finite element
method. We present our conclusions in section \ref{CONL}.

\section{Motivation: The finite strain incompressible elasticity problem}
\label{TFSI}
\setcounter{equation}{0}
To study the finite strain incompressible elasticity problem, we adopt what is known as the material
description in this paper. Given a reference configuration $\Omega\subset R^d~(1\leq d\leq 3)$
for a d-dimensional bounded material body $\mathcal{B}$, the
deformation of $\mathcal{B}$ can be described by the map
$\hat{\varphi}: \Omega \rightarrow  R^d$ defined by
$$
\hat{\varphi}(X)=X+\hat{u}(X),
$$
where~$X=(X_1,\cdots ,X_d)$ denotes the coordinates of a material
point in the reference configuration and $\hat{u}(X)$ represents the
corresponding displacement vector. Following the standard notation (see
\cite{auricchio2010importance}), we introduce the deformation
gradient $\hat{F}$ and the right Cauchy--Green deformation tensor
$\hat{C}$ by setting
\begin{equation}\label{E1}
\hat{F}=I+\nabla \hat{u},~~~~\hat{C}=\hat{F}^{T}\hat{F},
\end{equation}
where I is the second-order identity tensor and $\nabla$ is the
gradient operator with respect to the coordinate $X$.

For a homogeneous neo-Hookean material, for example, latex and rubber, we define (see \cite{bonet1997nonlinear}) the potential
energy function as
\begin{equation}\label{E2}
\psi(\hat{u})=\frac{1}{2}\mu[I:\hat{C}-d]-\mu
\ln\hat{J}+\frac{\lambda}{2}(\ln\hat{J})^2,
\end{equation}
where $\lambda$ and $\mu$ are positive constants, ``:''represents
the usual inner product for second-order tensors, and
$\hat{J}=\det\hat{F}$ is the deformation gradient Jacobian.

When we introduce he pressure-like variable (or simply pressure)
$\hat{p}=\lambda \ln\hat{J}$, the potential energy (\ref{E2}) can be
equivalently written as the following function of $\hat{u}$ and
$\hat{p}$ (still denoted as $\psi$ with a little abuse of notation):
$$
\psi(\hat{u},\hat{p})=\frac{1}{2}\mu[I:\hat{C}-d]-\mu
\ln\hat{J}+\hat{p}\ln\hat{J}-\frac{1}{2\lambda}(\hat{p})^2.
$$

When the body $\mathcal{B}$ is subject to a given load $b = b(X)$
per unit volume in the reference configuration, the total elastic
energy functional reads as
\begin{equation}\label{E3}
\Pi(\hat{u},\hat{p})=\displaystyle\int_\Omega \psi(\hat{u},\hat{p}) -\displaystyle\int_\Omega
b\cdot \hat{u}.
\end{equation}

According to the standard variational principles, the equilibrium is
derived by searching for critical points of (\ref{E3}) in suitable
admissible displacement and pressure spaces $\hat{V}$ and
$\hat{Q}$. The corresponding Euler--Lagrange equations arising from
(\ref{E3}) lead to this solution:
Find~$(\hat{u},\hat{p})\in \hat{V}\times\hat{Q}$~such~that
\begin{equation}\label{E4}
\left\{
\begin{array}{l}
\mu\displaystyle\int_\Omega\hat{F}:\nabla v+
\displaystyle\int_\Omega(\hat{p}-\mu)\hat{F}^{-T}: \nabla v=\displaystyle\int_\Omega
b\cdot v~~\forall v\in V,\\
                  \displaystyle\displaystyle\int_\Omega(\ln\hat{J}-\frac{\hat{p}}{\lambda})q=0~~\forall q\in
                  Q,
\end{array}
\right.
\end{equation}
where $V$ and $Q$ are the admissible variation spaces for the
displacements and pressures, respectively. Also note that in
(\ref{E4}), the linearization of the deformation
gradient Jacobian is
$$
DJ(\hat{u})[v]=J(\hat{u})F(\hat{u})^{-T}:\nabla
v=J(\hat{u})F(\hat{u}):\nabla v~~~\forall v\in V.
$$
We now focus on the case of an incompressible material, which
corresponds to taking the limit $\lambda\rightarrow +\infty$ in
(\ref{E4}). Hence, our problem becomes:
Find~$(\hat{u},\hat{p})\in \hat{V}\times\hat{Q}$~such~that
\begin{equation}\label{E5}
\left\{
\begin{array}{l}
\mu\displaystyle\int_\Omega\hat{F}:\nabla v+
\displaystyle\int_\Omega(\hat{p}-\mu)\hat{F}^{-T}: \nabla v=\displaystyle\int_\Omega
b\cdot v~~\forall v\in V,\\
                  \displaystyle\int_\Omega q\ln\hat{J}=0~~\forall q\in
                  Q,
\end{array}
\right.
\end{equation}
or, in residual form:
Find~$(\hat{u},\hat{p})\in \hat{V}\times\hat{Q}$~such~that
\begin{equation}
\left\{
\begin{array}{l}
\mathcal{R}_u((\hat{u},\hat{p}),v)=0~~\forall v\in V,\\
                  \mathcal{R}_p((\hat{u},\hat{p}),q)=0~~\forall q\in
                  Q,
\end{array}
\right.
\end{equation}
where
\begin{equation}
\left\{
\begin{array}{l}
\mathcal{R}_u((\hat{u},\hat{p}),v):=\mu\displaystyle\int_\Omega\hat{F}:\nabla v+
\displaystyle\int_\Omega(\hat{p}-\mu)\hat{F}^{-T}: \nabla v-\displaystyle\int_\Omega
b\cdot v,\\
                  \mathcal{R}_p((\hat{u},\hat{p}),q):=\displaystyle\int_\Omega
                  q\ln\hat{J}.
\end{array}
\right.
\end{equation}
We now derive the linearization of problem (\ref{E5}) around a
generic point $(\hat{u},\hat{p})$. Observing that
$$
D\hat{F}(\hat{u})[u]=-\hat{F}^{-T}(\nabla u)^T \hat{F}^{-T}~~\forall
u\in V,
$$
we easily get the problem for the infinitesimal increment $(u,p)$:
Find~$(u,p)\in V\times Q$~such~that
\begin{equation}\label{E6}
\left\{
\begin{array}{l}
\tilde{a}(u,v)+b(v,p)=-\mathcal{R}_u((\hat{u},\hat{p}),v)
~\forall v\in V,\\
                 \tilde{ b}(u,q)=-
                  \mathcal{R}_p((\hat{u},\hat{p}),q)~~\forall q\in
                  Q,
\end{array}
\right.
\end{equation}
where
\begin{equation}\label{E6}
\left\{
\begin{array}{l}
\tilde{a}(u,v)=\mu\displaystyle\int_\Omega\nabla u:\nabla v+
\displaystyle\int_\Omega(\mu-\hat{p})(\hat{F}^{-1}\nabla u)^T: \hat{F}^{-1}\nabla
v,\\
 \tilde{b}(u,q)=\displaystyle\int_\Omega q\hat{F}^{-T}: \nabla u.
\end{array}
\right.
\end{equation}
\begin{remark}
Since problem (\ref{E6}) is the linearization of problem (\ref{E5}),
 it can be interpreted as the generic step of a
Newton-like iteration procedure for the solution of the nonlinear
problem (\ref{E5}).
\end{remark}
\begin{remark}
Taking $(\hat{u},\hat{p})=(0,0)$ in (\ref{E6}), we immediately
recover the classical linear incompressible elasticity problem for
small deformations; i.e., we find~$(u,p)\in V\times Q$~such~that
\begin{equation}
\left\{
\begin{array}{l}
2\mu\displaystyle\int_\Omega\epsilon (u):\epsilon  (v)+\displaystyle\int_\Omega p~\hbox{div}
v=\int_\Omega b\cdot v
~\forall v\in V,\\
                  \displaystyle\int_\Omega q~\hbox{div} u=0~~\forall q\in
                  Q,
\end{array}
\right.
\end{equation}
where $\epsilon(u)=\frac{(\nabla u)^T+\nabla u}{2}$ denotes the
symmetric gradient operator.
\end{remark}

We now note that the Piola identity
$\hbox{div}(\hat{J}\hat{F}^{-T})=0$ and $\hat{J}=1$ give
$\hbox{div}\hat{F}^{-T}=0$. Hence, we have
$$
\hbox{div}(\hat{F}^{-T}u)=\hbox{div}\hat{F}^{-T}\cdot
u+\hat{F}^{-T}: \nabla u=\hat{F}^{-T}: \nabla u.
$$
Let $\hat{F}^{-T}u=w$, we then consider the following problem: Find
$(w,p)\in V\times Q$~such~that
\begin{equation}\label{E7}
\left\{
\begin{array}{l}
a(w,v)+b(v,p)=-\mathcal{R}_u((\hat{u},\hat{p}),v)
~\forall v\in V,\\
                  b(w,q)=-
                  \mathcal{R}_p((\hat{u},\hat{p}),q)~~\forall q\in
                  Q,
\end{array}
\right.
\end{equation}
where
\begin{equation}\label{E10}
\left\{
\begin{array}{l}
a(w,v)=\tilde{a}(\hat{F}w,\hat{F}v)=\mu\displaystyle\int_\Omega\nabla
(\hat{F}w):\nabla (\hat{F}v)+\\
~~~~~~~~~~~~~~~~~~~~~~~~~~~~~~\displaystyle\int_\Omega(\mu-\hat{p})(\hat{F}^{-1}\nabla
(\hat{F}w))^T: \hat{F}^{-1}\nabla
(\hat{F}v),\\
 b(w,p)=\displaystyle\int_\Omega
p\hbox{div} w.
\end{array}
\right.
\end{equation}

Let $V_h \subset V ~\hbox{and}~ Q_h \subset Q$ be the compatible
finite element spaces. The discrete problem of (\ref{E7}) reads:
Find~$(w_h,p_h)\in V_h\times Q_h$~such~that
\begin{equation}\label{E8}
\left\{
\begin{array}{l}
a(w_h,v_h)+b(v_h,p_h)=-\mathcal{R}_u((\hat{u},\hat{p}),v_h)
~\forall v_h\in V_h,\\
                  b(w_h,q_h)=-
                  \mathcal{R}_p((\hat{u},\hat{p}),q_h)~~\forall q_h\in
                  Q_h,
\end{array}
\right.
\end{equation}
where $a(\cdot,\cdot)~\hbox{and}~b(\cdot,\cdot)$~are defined by
(\ref{E10}).

 In some cases, it is possible that the continuous problem
(\ref{E7}) is stable (or well-posed) but that the discrete problem
(\ref{E8}) is not (see \cite{auricchio2005stability}). In such cases,
we say that there is an unphysical instability
caused by the discretization.

\section{Abstract stabilization strategy for mixed approximation}
\label{ASOM}
\setcounter{equation}{0}
As shown at the end of the previous section, it is necessary to establish stable
discretization when the continuous problem
(\ref{E7}) is stable. So in this section, we propose an abstract
stability strategy framework for the mixed approximation.

Let $V$ and $Q$ each be a Hilbert space, and assume that
$$
a: V \times V \rightarrow \mathbb{R},~b: V\times Q \rightarrow \mathbb{R}
$$
are continuous bilinear forms. Let $f \in V^{'}$ and $g \in Q^{'}$. We denote both the dual
pairing of $V$ with $V^{'}$ and that of $Q$ with $Q^{'}$ by $\langle\cdot,\cdot\rangle$.
We consider the following problem:
Find~$(w,p)\in V\times Q$~such~that
\begin{equation}\label{A0}
\left\{
\begin{array}{l}
a(w,v)+b(v,p)=\langle f,v \rangle
~~~~\forall v \in V,\\
                  b(w,q)=\langle g,q \rangle~~~~\forall q\in Q.
\end{array}
\right.
\end{equation}
We associate a mapping $B$ with the form $b$:
$$
B:V\rightarrow Q^{'}:\langle Bw,q \rangle=b(w,q)~~\forall q\in Q.
$$
Define
$$
K=\hbox{ker}(B):=\{v\in V: b(v,q)=0~~\forall q \in Q\}.
$$
As is well-known, the continuous problem (\ref{A0}) is well-posed
under the following assumptions:
\begin{itemize}
  \item $A_1:$ the inf-sup condition, i.e., it holds that
  \begin{equation}
\inf_{q\in Q}\sup_{v\in V} \frac{b(v,q)}{\|v\|_V\|q\|_Q}  \geq\beta>0,
  \end{equation}
where~$\|\cdot\|_{V} ~\hbox{and}~ \|\cdot\|_{Q}$~are norms on the
spaces~$V$~and~$Q$, respectively;
  \item $A_2:$ the coercivity on the kernel space, i.e., it holds that
   \begin{equation}\label{C1}
  a(v,v)\geq \alpha \|v\|_V^2 ~~ \forall~  v\in K, \hbox{for some}~ \alpha>0.
  \end{equation}
\end{itemize}


Now we rewrite the problem (\ref{A0}) in an equivalent form as follows:
Find~$(w,p)\in V\times Q$~such~that
\begin{equation}\label{A1}
\left\{
\begin{array}{l}
A(w,v)+b(v,p)=\langle f,v \rangle+M(g,Bv)_{Q^{'}}
~~~\forall v \in V,\\
                  b(w,q)=\langle g,q \rangle~~~\forall q\in Q,
\end{array}
\right.
\end{equation}
where $A(w,v)=a(w,v)+M (Bw,Bv)_{Q^{'}}, ~(\cdot,\cdot)_{Q^{'}}$ denotes the inner product on
$Q^{'}$, and $M$ is a parameter to be chosen properly.

Let $V_h\subset V$ and $Q_h\subset Q$ be finite dimensional subspaces of $V$ and $Q$ such that
the following assumption is satisfied:
\begin{itemize}
  \item $A_3:$~discrete inf-sup condition holds that
\begin{equation}\label{C2}
  \inf_{q_h\in Q_h}\sup_{v_h\in V_h}
  \frac{b(v_h,q_h)}{\|v_h\|_V\|q_h\|_Q}\geq  \beta_1>0,
  \end{equation}.
\end{itemize}
where $ \beta_1$ is independent of $h$.

We consider the discretization of the problem (\ref{A1}) as follows:
Find~$(w_h,p_h)\in V_h\times Q_h$~such~that
\begin{equation}\label{A4}
\left\{
\begin{array}{l}
A(w_h,v_h)+ b(v_h,p_h)=\langle f,v_h \rangle+M(g,Bv_h)_{Q^{'}}
~\forall v_h\in V_h,\\
                  b(w_h,q_h)=\langle g,q_h \rangle~~\forall q_h\in
                  Q_h.
\end{array}
\right.
\end{equation}
\begin{remark}
Note that when $M=0$, (\ref{A4}) becomes the classical mixed finite element method. We
call (\ref{A4}) the modified mixed finite element.
\end{remark}
\begin{lemma}\label{A2}(see \cite{brezzi1991mixed})
$A_1$ is equivalent to the following condition,
the operator $B:V^{\bot} \rightarrow Q^{'}$ is an isomorphism, and
\begin{equation}\label{A3}
\|Bv\|_{Q^{'}}\geq \beta \|v\|_V~~\forall v\in V^{\bot}.
\end{equation}
\end{lemma}

\begin{theorem}\label{N5} Under the assumptions $A_1$ and $A_2$, there is a constant $M_0$
such that the discrete problem (\ref{A4}) is stable for any $M\geq M_0$ in the sense that there exists a positive constant $\alpha_1 $ such that
\begin{equation}
 A(v_h,v_h)\geq \alpha_1 \|v_h\|_V^2 ~~\forall~v_h\in V_h.
\end{equation}
\end{theorem}
\begin{proof}
For any $v_h \in V_h$, noting that $V_h\subset V$,
we have $v_h=\varphi+\varphi^{\bot}$,
where $\varphi \in K, \varphi^{\bot} \in K^{\bot}$ and $K^{\bot}$ denotes the orthogonal of~$K$~in $V$.
 By Lemma \ref{A2}, we have
\begin{eqnarray*}
A(v_h,v_h)&=&a(v_h,v_h)+M (B v_h,B v_h)_{Q^{'}}\\
&=&a(\varphi+\varphi^{\bot},\varphi+\varphi^{\bot})+M (B \varphi^{\bot} ,B \varphi^{\bot})_{Q^{'}}\\
&\geq&a(\varphi,\varphi)+a(\varphi,\varphi^{\bot})+a(\varphi^{\bot},\varphi)
+a(\varphi^{\bot},\varphi^{\bot})+M\beta^2\|\varphi^{\bot}\|_V.
\end{eqnarray*}
By assumption $A_2$, the continuity of $a(\cdot,\cdot)$ and
the Cauchy inequality, we obtain that
\begin{eqnarray*}
A(v_h,v_h)&\geq & \alpha \|\varphi\|_V^2-
C_1\|\varphi\|_V\|\varphi^{\bot}\|_V
-C_2\|\varphi^{\bot}\|_V^2 +M\beta^2\|\varphi^{\bot}\|_V^2\\
&\geq &(\alpha -\varepsilon
C_1)\|\varphi\|_V^2-(C_2+\frac{C_1}{\varepsilon})\|\varphi^{\bot}\|^2_V+M\beta^2\|\varphi^{\bot}\|_V^2\\
&\geq &(\alpha -\varepsilon C_1)\|\varphi\|_V^2+(M\beta^2-(C_2+\frac{C_1}{\varepsilon}))
\|\varphi^{\bot}\|_V^2.\\
\end{eqnarray*}
Noting that $\|v_h\|^2_V=\|\varphi\|^2_V+\|\varphi^{\bot}\|^2_V$ and choosing
$\varepsilon=\frac{\alpha}{2C_1}, M_0=\displaystyle\frac{\frac{\alpha}{2}+C_2+\frac{2C_1^2}{\alpha}}{\beta^2}$, we have
\begin{eqnarray*}
A(v_h,v_h)\geq \alpha_1 (\|\varphi\|^2_V+\|\varphi^{\bot}\|^2_V)= \alpha_1
\|v_h\|_V^2,
\end{eqnarray*}
which completes the proof.
\end{proof}
Furthermore, noting that $A(\cdot,\cdot)$ is continuous, by the
classic Brezzi theory for the saddle point problem \cite{brezzi1991mixed,brezzi1974existence,arnold1984stable}, we have
\begin{theorem}\label{N3}
Let $w, p$ be the solution of the problem (\ref{A1}), then under the
assumptions $A_1,A_2,A_3$ and providing that $M$ is sufficiently large, the discrete problem
(\ref{A4}) has a unique solution $(w_h,p_h)\in V_h\times Q_h $, which
satisfies
\begin{equation}\label{A5}
\|w-w_h\|_V+\|p-p_h\|_Q \leq C \left( \inf_{v_h\in V_h}\|w-v_h\|_V+
\inf_{q_h\in Q_h}\|p-q_h\|_Q \right),
\end{equation}
where $C$ is a constant  depending on $\alpha_1$ and $\beta_1$.
\end{theorem}

\section{Application to the nonlinear elasticity problem}
\label{ATTN}
\setcounter{equation}{0}
We now apply the abstract stabilized framework proposed in the previous section to nonlinear incompressible
elasticity problems. Suppose that~$\Omega$~is connected, we consider the d-dimensional Dirichlet boundary
value problem (other boundary value problems are similar) of (\ref{E7}),  which
means that $V=(H_0^1(\Omega))^d,Q=L^2_0(\Omega)$. We choose
$V_h~\hbox{and}~Q_h$ as finite element spaces such that $A_3$ is satisfied.
Furthermore, $K=\{v\in (H_0^1(\Omega))^d: \hbox{div} v=0 \}$. 
We rewrite (\ref{E7}) as:~Find~$(w,p)\in V\times Q$~such~that
\begin{equation}\label{AE1}
\left\{
\begin{array}{l}
A(w,v)+b(v,p)=-\mathcal{R}_u((\hat{u},\hat{p}),v)-M
                  \mathcal{R}_p((\hat{u},\hat{p}),\hbox{div}v)
~~\forall v \in V,\\
                  b(w,q)=-\mathcal{R}_p((\hat{u},\hat{p}),q)~~\forall q\in Q,
\end{array}
\right.
\end{equation}
where $A(w,v)=a(w,v)+M (\hbox{div}w,\hbox{div}v)$, $a(\cdot,\cdot)~\hbox{and}~b(\cdot,\cdot)$
~are defined by
(\ref{E10}),~$(\cdot,\cdot)$ denotes the inner product on
$L^2(\Omega)$, and $M$ is a parameter chosen properly.

The discretization of the problem (\ref{AE1}) is as follows:
Find~$(w_h,p_h)\in V_h\times Q_h$~such~that
\begin{equation}\label{AE2}
\left\{
\begin{array}{l}
A(w_h,v_h)+b(v_h,p_h)=-\mathcal{R}_u((\hat{u},\hat{p}),v_h)-M
                  \mathcal{R}_p((\hat{u},\hat{p}),\hbox{div}v_h)
~~\forall v_h \in V_h,\\
                  b(w_h,q_h)=-\mathcal{R}_p((\hat{u},\hat{p}),q_h)~~\forall q_h\in
                  Q_h.
\end{array}
\right.
\end{equation}

\begin{lemma}(\cite{girault1986finite}, Corollary 2.3)\label{L1}
We have
$$
K^{\perp}=\{(-\triangle)^{-1}\hbox{grad} f: f \in L_0^2(\Omega)
\}.
$$
\end{lemma}
Let~$K^0=\{y \in (H^{-1}(\Omega))^d: <y,\phi>=0~~\forall ~\phi \in K\}$,~where~$<\cdot,\cdot>$~denotes the
duality paring between~$(H^{-1}(\Omega))^d$~and~$(H^{1}_0(\Omega))^d$.
\begin{lemma}(\cite{girault1986finite}, Corollary 2.4)\label{L2}
Let~$\Omega$~be connected. Then
\begin{enumerate}
  \item the operator \hbox{grad} is an isomorphism from $L^2_0{(\Omega)}$~onto $K^0$, and
  \item the operator \hbox{div} is an isomorphism from $K^{\perp}$~onto $L^2_0{(\Omega)}$.
\end{enumerate}
\end{lemma}
\begin{theorem} If $a(\cdot,\cdot)$
 defined by (\ref{E10}) satisfies
$$
a(v,v)\geq \alpha \|v\|_1^2 ~~ \forall~  v\in K,~ \hbox{for some}~ \alpha>0,
$$
the discrete problem (\ref{AE2}) is stable in the sense that
there exists a positive constant $\alpha_1 $ such that
\begin{equation}
 A(v_h,v_h)\geq \alpha_1 \|v_h\|_1^2~~~\forall  v_h\in V_h.
\end{equation}
\end{theorem}
\begin{proof}
By the abstract framework, the proof is obvious.
\end{proof}
Also from the abstract framework in the previous section, we can get an approximate
accuracy result similar to (\ref{A5}).
\begin{theorem}
Let $w, p$ be the solution of the problem (\ref{E7}), then under the assumptions $A_2$ and $A_3$ 
and providing that $M$ is sufficiently large, the discrete problem
(\ref{AE2}) has a unique solution $(w_h,p_h)\in V_h\times Q_h $, which
satisfies
\begin{equation}
\|w-w_h\|_1+\|p-p_h\|_0 \leq C \left( \inf_{v_h\in V_h}\|w-v_h\|_1+
\inf_{q_h\in Q_h}\|p-q_h\|_0 \right),
\end{equation}
where the constant $C$ depends on $\alpha_1, \beta_1$.
\end{theorem}
\section{Numerical experiment}
\label{NUEX}
\setcounter{equation}{0}
In this section, we report our numerical experiments in regard to the performance of the
stabilization strategy for two simple examples. We consider here simple problems
using some mixed finite element formulations that are known
to be optimal for Stokes equations. We first briefly present the finite element under consideration
and then show the numerical results obtained using such
element.
\subsection{Two examples}\label{TS}
In this subsection, we present two simple problems that will be
used in subsection \ref{NR} to discuss the performance of the stabilization strategy
proposed in Section 3. Using the usual Cartesian
coordinates $(X,Y)$, we consider a square material
body whose reference configuration is $\Omega = (-1,1) \times (-1, 1)$. We denote 
$\Gamma = [-1,1] \times \{1\}$ as the upper part of its boundary, while
the remaining part of $\partial\Omega$ is denoted with $\Gamma_{D}$.
The total energy is assumed to be as in (\ref{E3}), where the external
loads are given by the vertical uniform body forces: $b=\gamma f$,
where $f=(0,1)^{\top}$.
The two problems differ in regard to the imposed boundary conditions. More precisely:
\begin{itemize}
  \item \textbf{Problem 1.} We set clamped boundary conditions on $\Gamma_{D}$,
  but traction-free boundary conditions on $\Gamma$.
  \item \textbf{Problem 2.} We set vanishing normal displacements on $\Gamma_{D}$,
 but traction-free boundary conditions on $\Gamma$.
\end{itemize}
It is easy to see that both problems admit a trivial solution
for every $\gamma \in R,i.e.~(\hat{u},\hat{p})=(0,\gamma r)$, where $r=r(X,Y)=1-Y$.

For the problems under investigation, the corresponding linearized
problems (cf. (\ref{E10})) can both be written as:
\hbox{Find}~$(w,p)\in V\times Q$~such~that
\begin{equation}\label{N2}
\left\{
\begin{array}{l}
2\mu\displaystyle\int_\Omega\epsilon (w):\epsilon  (v)-\gamma
\displaystyle\int_\Omega r(\nabla w)^{\top}:\nabla (v)+\displaystyle\int_\Omega p~\hbox{div}
v=\delta \gamma \int_\Omega f\cdot v,
~\forall v\in V,\\
                  \displaystyle\int_\Omega q~\hbox{div} w=0~~\forall q\in Q.
\end{array}
\right.
\end{equation}
where $\delta \gamma$ is the increment of the parameter of $\gamma$.

For these two different problems,
the spaces $V$ and $Q$ are defined as follows:
\begin{itemize}
  \item \textbf{Problem 1.} $V=\{v\in H^1(\Omega)^2:~v|_{\Gamma_D}=0\}; Q=L^2(\Omega)$.
  \item \textbf{Problem 2.} $V=\{v\in H^1(\Omega)^2:~(v\cdot n)|_{\Gamma_D}=0\}; Q=L^2(\Omega)$,
  where $n$ denotes the outward normal vector.
\end{itemize}
The stable discrete formulation reads as follows (see (\ref{AE2})):
Find~$(w_h,p_h)\in V_h\times Q_h$~such~that
\begin{equation}\label{N1}
\left\{
\begin{array}{l}
A(w_h,v_h)+b(v_h,p_h)=\delta \gamma \int_\Omega f\cdot v_h,
~\forall v_h \in V_h,\\
                  b(w_h,q_h)=0,~~\forall q_h\in
                  Q_h.
\end{array}
\right.
\end{equation}
where $A(w,v)=2\mu\displaystyle\int_\Omega\epsilon (w):\epsilon  (v)-\gamma
\displaystyle\int_\Omega r(\nabla w)^{\top}:\nabla (v)+M(\gamma) \int_\Omega\hbox{div}w~\hbox{div}v$;\\
$~b(w,p)=\int_\Omega p~\hbox{div}v$. Here $M(\gamma)$ is a parameter chosen
 depending on $\gamma$.
\begin{remark}
For Problem $1$, it has been theoretically proved in \cite{auricchio2005stability}
that the continuous problem (\ref{N2}) is stable when $\gamma<3\mu$.
\end{remark}

\subsection{The MINI element}
The considered scheme for (\ref{N1}) is the MINI element (see \cite{arnold1984stable}). Let $T_h$
be a triangular mesh of $\Omega$ with the mesh size $h$. For the discretization
of the displacement field, we take
$$
V_h=\{v_h \in V: v_h|_{T} \in P_1(T)^2+B(T)^2~\forall~T \in T_h\},
$$
where $P_1(T)$ is the space of linear functions on $T$, and $B(T)$ is the linear
space generated by $b_T$, the standard cubic bubble function on $T$.
For the pressure discretization, we take
$$
Q_h=\{q_h \in H^1(\Omega)\cap Q : q_h|_{T} \in P_1(T)~\forall~T \in T_h\}.
$$
\subsection{Numerical results}
\label{NR}
\setcounter{equation}{0}
We now study the stability performance of the discretized model problems
by means of the modified mixed finite element formulations briefly
described above. It has been theoretically proved and numerically verified
in \cite{auricchio2005stability} that for Problem $1$ the classical mixed finite element method is stable
when $\gamma <\mu$, but unstable when $\gamma > \frac{3}{2}\mu$.
In this numerical experiment, we will demonstrate that the modified mixed finite element
method is stable for both Problem $1$ and Problem $2$ when $- \infty <\gamma< 3 \mu$ (which is the stability range for the continuous case of Problem $1$ in \cite{auricchio2005stability}) as predicated by our Theorem \ref{N5}.

Noting Theorem \ref{N5}, 
we study the eigenvalues
of the matrix induced by the bilinear form $A(\cdot,\cdot)$
for the problems under consideration. The first loads for which we find a negative eigenvalue are the critical ones. 
We start from $\gamma=0$ for both positive and negative loading conditions, i.e., for $\gamma<0$ and $\gamma>0$.
We indicate as the critical loads $\gamma_{m,h}$ and $\gamma_{M,h}$
the first load values for which we find a negative eigenvalue.
A subsequent bisection-type procedure is used to increase the accuracy
of the critical load detections. The corresponding nondimensional
quantities are denoted with $\tilde{\gamma}_{m,h}$ and $\tilde{\gamma}_{M,h}$, respectively,
where $\tilde{\gamma}=\frac{\gamma L}{\mu}$. Here, $L$ is some problem characteristic
length, set equal to $1$ for simplicity, consistents also with the geometry of the model
problems. If we do not detect any negative eigenvalue 
for extremely large values of the load multiplier $(\tilde{\gamma}> 10^6)$, we set $\tilde{\gamma}=\infty$. 
Noting that the bound constant $C_1$ for the formulation
$a(\cdot,\cdot)$ is of order $\gamma+2 \mu $ and the stable constant of
$a(\cdot,\cdot)$ is of order $\mu$, we can choose
$M_0=m_1 |\tilde{\gamma}|+m_2 \tilde{\gamma}^2$ by the proof of Theorem \ref{N5}. Hence, in the codes, we set
$\mu=40, M(\gamma)=m_1 |\tilde{\gamma}|+m_2 \tilde{\gamma}^2$,
then we adjust the parameter $\tilde{\gamma}$, $m_1$ and $m_2$ 
to investigate the stability performance. Furthermore, we take $m_1=320,m_2=0$ for Problem $1$
and $m_1=320,m_2=1.36$ for Problem $2$.

\begin{table}[h]
\centering\caption
{
 Stability Limits for Problem $1$~
 }
\begin{tabular}{|c|c|c|}
\hline
 Nodes ~~~~~~& $\tilde{\gamma}_{m,h}$ & ~~~~~~$\tilde{\gamma}_{M,h}$ ~~~~~~\\
 $5\times 5$ ~~~~~~&~~~~~~~ $-\infty$ ~~~~~~&~~~~~~$+\infty$ ~~~~~~\\
 $9\times 9$ ~~~~~~&~~~~~~ ~$-\infty$~~~~~~ &~~~~~~$14.50$ ~~~~~~\\
 $17\times 17$~~~~~~ & ~~~~~~~$-\infty$~~~~~~ &~~~~~~$8.25$~~~~~~ \\
 $33\times 33$ ~~~~~~& ~~~~~~~$-\infty$~~~~~~ &~~~~~~$7.13$ ~~~~~~\\
 \hline
 \end{tabular}
\end{table}


\begin{table}[h]
\centering\caption
{
 Stability Limits for Problem $2$~
 }
\begin{tabular}{|c|c|c|}
\hline
 Nodes ~~~~~~~& $\tilde{\gamma}_{m,h}$ &~~~~~~$\tilde{\gamma}_{M,h}$ ~~~~~~\\
 $5\times 5$ ~~~~~~~&~~~~~~~ $-\infty$~~~~~~&~~~~~~$+\infty$ ~~~~~~\\
 $9\times 9$ ~~~~~~~&~~~~~~~ $-\infty$~~~~~~ &~~~~~~$3.88$ ~~~~~~\\
 $17\times 17$~~~~~~~ &~~~~~~~$-\infty$~~~~~~ &~~~~~~$3.38$~~~~~~ \\
 $33\times 33$ ~~~~~~~& ~~~~~~~$-\infty$~~~~~~ &~~~~~~$3.23$ ~~~~~~\\
 \hline
 \end{tabular}
\end{table}

To verify the convergence result (\ref{N3}), we set another data for (\ref{N2}) that
$f=(-e^{x}(1-y),e^{x})^{\top}$ and the true solution is $(w,p)=(0,\delta\gamma e^{x}(1-y))$.
\begin{table}[h]
\centering\caption
{The Convergence of $\tilde{\gamma}=7.125$ for Problem $1$}
\begin{tabular}{|c|c|c|c|}
  \hline
Nodes & $\|p-p_h\|_{0}$ & $\|w-w_h\|_{1}$ &$order$ \\
$5\times 5$ &$6.0469\times 10^{-2}$ &$1.0518\times 10^{-6}$ &$--$\\
$9\times 9$ &$1.5141\times 10^{-2}$ &$1.8890\times 10^{-7}$ &$2$\\
$17\times 17$ &$3.7871\times 10^{-3}$ &$3.0897\times 10^{-8}$ &$2$\\
$33\times 33$  &$9.4692\times 10^{-4}$ &$9.2283\times 10^{-9}$ &$2$\\
  \hline
\end{tabular}
\end{table}
\begin{table}[h]
\centering\caption
{The convergence of $\tilde{\gamma}=3.23$ for Problem $2$}
\begin{tabular}{|c|c|c|c|}
  \hline
Nodes & $\|p-p_h\|_{0}$ & $\|w-w_h\|_{1}$ &$order$ \\
$5\times 5$ &$6.0187\times 10^{-2}$ &$7.5563\times 10^{-6}$ &$--$\\
$9\times 9$ &$1.5129\times 10^{-2}$ &$1.6314\times 10^{-6}$ &$2$\\
$17\times 17$ &$3.7867\times 10^{-3}$ &$3.3616\times 10^{-7}$ &$2$\\
$33\times 33$  &$9.4691\times 10^{-4}$ &$7.7049\times 10^{-8}$ &$2$\\
  \hline
\end{tabular}
\end{table}

Table 1 and Table 2 show that the stabilization strategy together with the corresponding
modified mixed finite element method proposed in this
paper is effective. The stability performance can be improved obviously. In fact,
these values of $\tilde{\gamma}_{M,h}$ are competitive to the ``exact" values claimed in
\cite{auricchio2010importance}.
Furthermore, we can see that the modified mixed finite
element method is also locking-free by verifying the convergence results
(\ref{N3}), which are shown in Table 3 and Table 4. In fact, the classical
mixed finite element method is unstable for the Problem $1$
when $\tilde{\gamma}=7.125$, and hence the convergence fails.
However, the modified method performs well.

\section{Conclusions}
\label{CONL}
\setcounter{equation}{0}
Within the framework of finite elasticity for incompressible materials, it is
well known that the classical mixed finite element discretization can sometimes be unstable even
though the continuous problem is stable. In this paper, we reformulated the continuous
problem and proposed an abstract stabilization strategy based on the new continuous
formulation and obtained a modified mixed finite element method.
We proved theoretically in Section \ref{ASOM} that for a sufficiently large $M$, the modified mixed finite element method
is stable whenever the continuous problem is stable, and the method maintains the
optimal convergence of the classical one. We verified by numerical experiments
in Section \ref{NUEX} that the modified mixed finite element method is much more stable
than the classical one and is also locking-free. The $M_0$ provided in Section
\ref{ASOM} always overestimates the parameter used in practical problems.
However, we can choose the parameter heuristically by analyzing the stability and 
continuity of continuous problems in the numerical experiments presented 
in Section \ref{NUEX}. 

\bibliographystyle{elsarticle-num}
\bibliography{nonlinearelasticity}

\end{document}